\documentclass[a4paper,12pt]{amsart}
\usepackage{amssymb}
\usepackage{ifthen}
\usepackage[dvips]{graphicx}
\usepackage{subcaption}
\nonstopmode \numberwithin{equation}{section}
\usepackage{amssymb}
\usepackage{ifthen}
\usepackage{graphicx}
\usepackage{amsmath}
\usepackage[T1]{fontenc} %skandit
\usepackage[utf8]{inputenc}
\usepackage[usenames,dvipsnames]{color}
\usepackage{color}
\usepackage[english]{babel}
\usepackage{fancyhdr}
\usepackage{fancybox}
\usepackage{tikz}

\nonstopmode \numberwithin{equation}{section}
\newtheorem{theorem}{Theorem}[section]

\theoremstyle{remark}
\theoremstyle{definition}

\newtheorem{definition}{Definition}[section]

\theoremstyle{plain}
\newtheorem*{thmA}{Theorem A}

\newtheorem*{lemA}{Lemma A}
\newtheorem*{lemB}{Lemma B}

\numberwithin{equation}{section}
\numberwithin{theorem}{section}

\newcounter{minutes}
\divide\time by 60
\newcounter{hours}
\multiply\time by 60
\addtocounter{minutes}{-\time}
\begin{document}

\title{Sharp radius of concavity for certain classes of analytic functions}

\author{Molla Basir Ahamed}
\address{Molla Basir Ahamed, Department of Mathematics, Jadavpur University, Kolkata-700032, West Bengal, India.}
\email{mbahamed.math@jadavpuruniversity.in}

\author{Rajesh Hossain}
\address{Rajesh Hossain, Department of Mathematics, Jadavpur University, Kolkata-700032, West Bengal, India.}
\email{rajesh1998hossain@gmail.com}

\subjclass[2020]{Primary 30C45, 30C55, 30C80}
\keywords{Radius of concavity, Radius problem, Analytic functions, Univalent functions, Starlike functions, Convex functions}

\def\thefootnote{}
\footnotetext{ {\tiny File:~\jobname.tex,
printed: \number\year-\number\month-\number\day,
          \thehours.\ifnum\theminutes<10{0}\fi\theminutes }
} \makeatletter\def\thefootnote{\@arabic\c@footnote}\makeatother

\begin{abstract}
Let $\mathcal{A}$ be the class of all analytic functions $f$ defined on the open unit disk $\mathbb{D}$ with the normalization $f(0)=0=f^{\prime}(0)-1$. This paper examines the radius of concavity for various subclasses of $\mathcal{A}$, namely $\mathcal{S}_0^{(n)}$, $\mathcal{K(\alpha,\beta)}$, $\mathcal{\tilde{S^*}(\beta)}$, and $\mathcal{S}^*(\alpha)$. It also presents results for various classes of analytic functions on the unit disk.  All the radii are best possible.
\end{abstract}
\maketitle
\pagestyle{myheadings}
\markboth{M. B. Ahamed and R. Hossain}{Sharp radius of concavity for certain class of analytic functions}

\tableofcontents
\section{\bf Introduction}
Throughout the article, we will denote the open unit disk of the complex plane $ \mathbb{C} $ by $ \mathbb{D}:=\{z\in\mathbb{C} : |z|<1\} $. Let $\mathcal{A}$ be the class of functions $f$ analytic in the unit disk $\mathbb{D}$ and normalized by $f(0)=0$ and $f^{\prime}(0)=1$. We denote by $\mathcal{S}$ the class of univalent functions in $\mathcal{A}$. For $\alpha\leq 1$, let $S^*(\alpha)$ denote the subclass of $\mathcal{A}$ of starlike functions of order $\alpha$. If $0\leq \alpha<1$, then $f$ is said to be univalent and starlike of order $\alpha$, \textit{i.e.,} we have $\mathcal{S}^*_{\alpha}\subset \mathcal{S}$. A function $f\in\mathcal{A}$ is said to be starlike of order $\alpha$ if, and only if, 
\begin{align*}
	{\rm Re} \left(\frac{z f^{\prime}(z)}{f(z)}\right)>\alpha.
\end{align*}
The class of starlike functions of order $\alpha$ is denoted by $S^*(\alpha)$.  \vspace{1.2mm}

 Let $\mathcal{C}_{\alpha}$ be the class of convex functions of order $\alpha$ corresponding to $\mathcal{S}_{\alpha}$ in the usual way that $g\in\mathcal{C}_{\alpha} \Leftrightarrow zg^{\prime}\in \mathcal{S}^*(\alpha)$. Note that $\mathcal{S}^*(0)=\mathcal{S}^*$ and $\mathcal{C}_{0}=\mathcal{C}$. It is well known that $f\in\mathcal{C}$ if, and only if
 \begin{align*}
 	{\rm Re}\left(1+\frac{z f^{\prime\prime}(z)}{f^{\prime}(z)}\right)>0,\; z\in\mathbb{D},
 \end{align*}
 and $f\in\mathcal{S}^*$ if, and only if
 \begin{align*}
 	{\rm Re}\left(\frac{z f^{\prime}(z)}{f(z)}\right)>0,\; z\in\mathbb{D}.
 \end{align*}
 Further, it is worth mentioning that $\mathcal{C}\subsetneq\mathcal{S}^*$. Interestingly, each $f\in\mathcal{S}$ maps $\mathbb{D}_r:=\{z\in\mathbb{C} : |z|<r\}$, $r\in (0, 1]$ onto a convex domain if, and only if, ${\rm Re}\left(1+zf^{\prime\prime}(z)/f^{\prime}(z)\right)>0$ for $z\in\mathbb{D}_r$ (see \cite[p. 105, Problem 108]{Pólya-Szeg˝o-1954}).\vspace{1.2mm}
 
 In $1991$, Goodman (see \cite{Goodman-UCF-1971,Goodman-USF-1971}) introduced the concepts uniform convexity and uniform starlikeness. A function $f\in\mathcal{S}$ is said to be uniformly convex (starlike) if for every ciruclar arc $\gamma$ contained in $\mathbb{D}$, with center $\zeta_0\in\mathbb{D}$, the image arc $f(\gamma)$ is convex (starlike with respect to $f(\zeta_0)$). Goodman denoted the class of uniformly convex functions by $UCV$ and the class of uniformly starlike functions by $UST$.  \vspace{1.2mm}

Let $f$ and $g$ be analytic in $\mathbb{D}$. Then we say that $f$ is subordinate to $g$ in $\mathbb{D}$, written by $f\prec g$, if there exists a function $\omega(z)$ analytic in $\mathbb{D}$ which satisfies $\omega(0)=0$, $|\omega(z)|<1$ and $f(z)=g(\omega(z))$ for all $z\in\mathbb{D}$. If $g$ is univalent in $\mathbb{D}$, then the subordination $f\prec g$ is equivalent tot $f(0)=g(0)$ and $f(\mathbb{D})\subset g(\mathbb{D})$ (cf.  \cite{Duren-1983-NY}).\vspace{1.2mm}

We consider the class $ \mathcal{A}(p) $, where $ p\in (0, 1) $ consisting of all meromorphic functions in $ \mathbb{D} $ with a simple pole at $ z=p $ and normalized by the condition $ f(0)=0=f^{\prime}(0)-1$. Let $ \mathcal{S}(p) $ be the class of all univalent functions in $ \mathcal{A}(p) $. In $ 1931 $, Fenchel \cite{Fenchel-1931} obtained the sharp lower bound of $ |f(z)| $ for $ f\in\mathcal{S}(p) $ and the upper bound was established by Kirwan and Schober (see \cite{Kirwan-Schober-JAM-1976}) in $ 1976 $. It is well-known that if $ f\in\mathcal{S}(p) $, then 
\begin{align}\label{Eq-1.1}
	|k_p(-r)|\leq |f(z)|\leq |k_p(z)|,\;\; |z|=r<1,
\end{align}
where 
\begin{align*}
	k_p(z):=\frac{-pz}{(z-p)(1-pz)}\in \mathcal{S}(p)\; \mbox{for}\; z\in\mathbb{D}.
\end{align*}
Let $ {\rm Co}(p) $ be a subclass of $ \mathcal{S}(p) $ consisting of functions $ f $ such that $ \overline{\mathbb{D}}\setminus f(\mathbb{D}) $ is a bounded convex set, where $ \overline{\mathbb{C}}:=\mathbb{C}\cup\{\infty\} $. In \cite{Pfaltzgraff-JAM-1971}, it is proved that $ f\in {\rm Co}(p) $ if, and only if, $ f\in\mathcal{S}(p) $ and there exists a holomorphic function $ P_f $ in $ \mathbb{D} $ such that 
\begin{align*}
	{\rm Re}\; P_f(z)>0,\; z\in\mathbb{D},\; P_f(p)=\frac{1+p^2}{1-p^2}\; \mbox{and}\; P_f(0)=1,
\end{align*}
where
\begin{align}\label{Eq-1.2}
	P_f(z)=-\left(1+\frac{zf^{\prime\prime}(z)}{f^{\prime}(z)}+\frac{z+p}{z-p}-\frac{1+pz}{1-pz}\right).
\end{align}
Recently, in \cite{Bhowmik-CMFT-2024} the radius of concavity is defined. Henceforth, it is worth observing that $ g(z):=r^{-1}f(rz) $, $ z\in\mathbb{D} $ does not belong to $ {\rm Co}(p) $, whenever $ f\in {\rm Co}(p) $. Consequently, it cannot be concluded that if $ \overline{\mathbb{D}}\setminus f(\mathbb{D}) $ is convex, then $ \overline{\mathbb{D}}\setminus f(\mathbb{D}_r) $ is convex for each $ 0<r\leq 1 $. Because of this fact, the radius of concavity (w.r.t $ {\rm Co}(p) $) of subset of $ \mathcal{A}(p) $ in the following way.
\begin{definition}\cite{Bhowmik-CMFT-2024}\label{Def-1.1}
	The radius of concavity (w.r.t $ {\rm Co}(p) $) of a subset $ \mathcal{A}_1(p) $ of $ \mathcal{A}(p) $ is the largest number $ R_{{\rm Co}(p)}\in (0,1] $  such that for each function $ f\in\mathcal{A}_1(p) $, $ {\rm Re} \left(P_f(z)\right)>0 $ for all $ |z|< R_{\rm Co(p)}$, where $ P_{f}(z) $ is defined in \eqref{Eq-1.2}.
\end{definition}
In this paper we find a lower bound of the radius of concavity $ R_{\rm Co(p)} $ of  the class $ \mathcal{S}(p) $. Now, we consider functions $f$ in $ \mathcal{A} $ that map $ \mathbb{D} $ conformally onto a domain whose complement with respect to $ \mathbb{C} $ is convex and that satisfy the normalization $ f(1)=\infty $. We will denote these families of functions by $ {\rm Co}(A) $. Now $ f\in {\rm Co}(A) $ if, and only if, $ {T_f(z)}>0 $ for every $ z\in\mathbb{D} $, where $ f(0)=f^{\prime}(0)-1 $ and 
\begin{align}\label{Eq-1.3}
T_f(z)=\frac{2}{A-1}\left(\frac{(A+1)}{2}\frac{1+z}{1-z}-1-z\frac{f^{\prime\prime}(z)}{f^{\prime}(z)}\right).
\end{align}
\begin{definition}\cite{Bhowmik-CMFT-2024}\label{Def-1.2}
The radius of concavity (w.r.t $ {\rm Co(A)} $) of a subset $ \mathcal{A}_1 $ of $ \mathcal{A} $ is the largest number $ \mathrm{R}_{\mathrm{Co(A)}}\in (0,1] $  such that for each function $ f\in\mathcal{A}_1 $, $ {\rm Re} \left(T_f(z)\right)>0 $ for all $ |z|< \mathrm{R}_{\mathrm{Co(A)}}$, where $ T_{f}(z) $ is defined in \eqref{Eq-1.3}.
\end{definition}
Inspired by \cite[Definition 1.1.]{Bhowmik-CMFT-2024}, for an arbitrary family $\mathcal{F}$ of functions, we define the radius of concavity.
\begin{definition}
	The radius of concavity (w.r.t $ \mathcal{F} $), a subclass of $ \mathcal{A} $ is the largest number $ \mathrm{R}_{\mathcal{F}}\in (0,1] $  such that for each function $ f\in\mathcal{F} $, $ {\rm Re} \left(T_f(z)\right)>0 $ for all $ |z|< \mathrm{R}_{\mathcal{F}}$, where $ T_{f}(z) $ is defined in \eqref{Eq-1.3}.
\end{definition}
\subsection{\bf Motivation}
Determining the radius of concavity for a certain class of analytic functions is currently an active area of research in geometric function theory. However, there is limited focus in the literature on exploring the radius of concavity for specific classes of functions. In this paper, we were inspired by the results of Bhowmik and Biswas \cite{Bhowmik-CMFT-2024} and obtained a result that finds the radius of concavity for certain sub-classes of the class $\mathcal{A}$. The main results and their related background will be discussed in the subsequent sections.
\section{\bf Radius of concavity for functions in the class $\mathcal{S}_0^{(n)}$}
For $-1\leq B<A\leq 1$ and $p(z)=1+c_nz^n+c_{n+1}z^{n+1}+\cdots$, $n\in\mathbb{N}$, we say that $p\in P_n[A, B]$ if 
\begin{align*}
	p(z)\prec \frac{1+Az}{1+Bz},\;z\in\mathbb{D}.
\end{align*}
The class of functions $f\in\mathcal{A}$ with the property that $zf^{\prime}(z)/f(z)\in P_n[A, B]$ is denoted by $ST_n[A, B]$. If $n=1$, we drop the subscript. We see that $ST[1-2\alpha, -1]=\mathcal{S}_{\alpha}$, and we denote $ST_n[1-2\alpha, -1]=P_n(\alpha)$, and the special case $P_n(0)$, we denote simply by $P_n$. 
\begin{lemA}\emph{(see \cite[Lemma 2.1]{Ravi-Ron-Shan-CVEE-2007})}
	If $p\in P_n[A, B]$, then 
	\begin{align*}
		\bigg|p(z)-\frac{1-ABr^{2n}}{1-B^2r^{2n}}\bigg|\leq \frac{(A-B)r^n}{1-B^2r^{2n}},\; |z|=r<1.
	\end{align*}
	For the special case $p\in P_n(\alpha)$, we get
	\begin{align*}
		\bigg|p(z)-\frac{1+(1-2\alpha)r^{2n}}{1-r^{2n}}\bigg|\leq\frac{2(1-\alpha)r^n}{1-r^{2n}},\; |z|=r<1.	\end{align*}
\end{lemA}
We recall here a result by MacGregor that will be used later.
\begin{lemB}
	If $p\in {P}_n$. then 
	\begin{align*}
		\bigg|\frac{z p^{\prime}(z)}{p(z)}\bigg|\leq\frac{2n r^n}{1-r^{2n}}\; \mbox{for}\; |z|=r<1.
	\end{align*}
\end{lemB}
\begin{thmA}\emph{(see \cite[Theorem 2.6]{Ravi-Ron-Shan-CVEE-2007})}
	The radius of convexity of order $\beta$, $0\leq\beta<1$, in $\mathcal{S}_0^{(n)}$ is 
	\begin{align*}
		R(\beta)=\bigg[\frac{(1+n)-\sqrt{n^2+2n+\beta^2}}{1+\beta}\bigg]^{1/n}
	\end{align*}
	and the radius of uniform convexity in $\mathcal{S}_0^{(n)}$ is $R_{UCV}=R(1/2)$. 
\end{thmA}
We determine the radius of concavity for the class $\mathcal{S}^{(n)}_0$ and proved the result. 
\begin{theorem}
	If $f\in\mathcal{S}^{(n)}_0$,  then $\rm{Re}\left(T_f(z)\right)>0$ for $|z|< \mathrm{R_{n,{Co(A)}}}$ , where $\mathrm{R_{n,{Co(A)}}}$ is the least value of $r\in(0,1)$ satisfying $\Phi_1(r)=0$ with 
	\begin{align*}
		\Phi_1(r):=\frac{A+1}{2}\left(\frac{1-r}{1+r}\right)-\frac{1+2(n+1)r^n+r^{2n}}{1-r^{2n}},\; |z|=r<1.
	\end{align*}
	The radius $\mathrm{R_{n,{Co(A)}}}$ is sharp. 
\end{theorem}
\begin{proof}
	Since $f\in\mathcal{S}_0^{(n)}$, we have $p(z)=zf^{\prime}(z)/f(z)\in P_n$. Then in view of Lemma A, we have 
	\begin{align*}
		\bigg|\frac{z f^{\prime}(z)}{f(z)}-\frac{1+r^{2n}}{1-r^{2n}}\bigg|\leq \frac{2 r^n}{1-r^{2n}}\; \mbox{for}\; |z|=r<1.
	\end{align*}
	Further from Lemma B, we have 
	\begin{align*}
		\bigg|\frac{z p^{\prime}(z)}{p(z)}\bigg|\leq\frac{2n r^n}{1-r^{2n}}\; \mbox{for}\; |z|=r<1.
	\end{align*}
	
\noindent A simple computation shows that
	\begin{align*}
	1+\frac{z f^{\prime}(z)}{f(z)}=p(z)+\frac{z p^{\prime}(z)}{p(z)},
	\end{align*}
	hence, 
	\begin{align}\label{Eq-2.1}
	\bigg|\frac{zf^{\prime\prime}(z)} {f^{\prime}(z)}-\frac{2 r^{2n}}{1-r^{2n}}\bigg|\leq\frac{2(n+1)r^n}{1-r^{2n}}.
	\end{align}
	From \eqref{Eq-2.1}, it follows that
	\begin{align}\label{Eq-2.2}
		{\rm Re}\left(1+\frac{zf^{\prime\prime}(z)} {f^{\prime}(z)}\right)&\leq 1+\frac{2(n+1)r^n+r^{2n}}{1-r^{2n}},\;\;\;|z|=r<1.\\&\nonumber=\frac{1+2(n+1)r^n+r^{2n}}{1-r^{2n}}.
	\end{align}
	In view of the inequality \eqref{Eq-2.2}, we see that
		\begin{align*}
			{\rm Re}\left(T_{f(z)}\right)\geq \frac{2}{A-1}\left[\frac{A+1}{2}\left(\frac{1-r}{1+r}\right)-\frac{1+2(n+1)r^n+r^{2n}}{1-r^{2n}}\right]\; \mbox{for}\;|z|=r<1.
		\end{align*}
		The right hand side of the above inequality is strictly positive if $|z|< \mathrm{R_{n,{Co(A)}}}$ , where  $\mathrm{R_{n,{Co(A)}}}$  is given in the statement of the theorem. We now investigate the existence of  $\mathrm{R_{n,{Co(A)}}}$  for each $A\in(1,2]$.\vspace{1.2mm}

	We see that the function $\Phi_1(r)$ which is defined in the statement of the theorem is continuous on $[0,1]$ with the properties
		\begin{align*}
		\Phi_1(0)=A-1>0\;\;\;\mbox{and}\;\;\; \lim_{r\rightarrow {1}^{-}}\Phi_1(r)=-\infty.
		\end{align*}
		By the Intermediate Value Theorem (IVT), $\Phi_1(r)$ has at least one root in $(0,1)$.  Hence, 	${\rm Re}\left(T_{f(z)}\right)>0$ if $|z|=r< \mathrm{R_{n,{Co(A)}}}$  exists for every $A\in(1,2]$. For $n\in\mathbb{N}$, if we consider the function
		\begin{align}
			f_0(z)=\frac{z}{(1-z^n)^{\frac{2}{n}}},\; z\in\mathbb{D},
		\end{align} then we see that
		\begin{align*}
			T_{f_0}(z)=\frac{2}{A-1}\left[\frac{A+1}{2}\left(\frac{1-z}{1+z}\right)-1-\frac{(2n+2)z^n+(1-n+\frac{2}{n})z^{2n}}{(1-z^{2n})}\right].
		\end{align*} 
		We observe that, if $z=-r$ and $\mathrm{R_{n,Co(A)}}<|z|<1$, then ${\rm Re}\;T_{f_0}(z)<0$ (see Figure \ref{fig-1} for some particular cases). This proves the sharpness of the radius of concavity.
\end{proof}
\section{\bf Radius of concavity for functions in the class $\mathcal{K(\alpha,\beta)}$}
Many classes of functions studied in geometric function theory can be unified through the so-called Kaplan classes. We give a brief introduction here to this topic, following the article \cite{Ruscheweyh-1982}. Let $\mathcal{A}_0$ be the class of functions $g$ with the property that $zg\in\mathcal{A}$. We define $\mathcal{H}$ to be the class of functions $f\in\mathcal{A}_0$ such that 
\begin{align*}
	{\rm Re}\; e^{i\delta} f(z)>0\; \mbox{for}\; z\in\mathbb{D}
\end{align*}
for some $\delta\in\mathbb{R}$.\vspace{1.2mm} 

For $\alpha>0$, we define the class
\begin{align*}
	\mathcal{H}^{\alpha}=\{f^{\alpha}\in\mathcal{A}_0 : f\in\mathcal{H}\}.
\end{align*}
Let $\mathcal{K}(0, \gamma)$, $\gamma\geq 0$, be the class of functions $g\in\mathcal{A}_0$ with ${\rm Re} \left(zg^{\prime}/g(z)\right)\geq -\gamma/2$ for $z\in\mathbb{D}$. Define (see \cite[p. 32]{Ruscheweyh-1982})
\begin{align*}
	\mathcal{K}(\alpha, \beta)=\mathcal{H}^{\alpha}\cdot K(0, \beta-\alpha)\; \mbox{if}\; 0\leq\alpha\leq \beta
\end{align*}
and 
\begin{align*}
	\mathcal{K}(\alpha, \beta)=\bigg\{g : g(z)=\frac{1}{f(z)},\; f\in\mathcal{K}(\beta, \alpha)\bigg\}\; \mbox{if}\; 0\leq\beta\leq\alpha.
\end{align*}
The classes $\mathcal{K}(\alpha, \beta)$ are called Kaplan classes of type $\alpha$ and $\beta$. The name is due the fact that the close-to-convex functions, $\mathcal{K}$, introduced by Kaplan in \cite{Kaplan-MMJ-1952}, have the property that $f\in\mathcal{K}\Leftrightarrow f^{\prime}\in \mathcal{K}(1, 3)$. It is worth pointing out that the definition of $\mathcal{K}(0, \gamma)$ implies that \begin{align*}
	g\in\mathcal{S}_{\alpha}\Leftrightarrow \frac{g}{z}\in\mathcal{K}(0, 2-2\alpha).
\end{align*}
Further, it can be shown (see \cite{Ruscheweyh-1982}) that 
\begin{align*}
	\mathcal{K}(\alpha, 0)\cdot\mathcal{K}(0, \beta)\subset \mathcal{K}(\alpha, \beta).
\end{align*}
We obtain the result finding the radius of concavity of the class $\mathcal{K(\alpha,\beta)}$. 
\begin{theorem}
	If $f \in\mathcal{K(\alpha,\beta)}$, $0\leq\alpha\leq\beta$ then ${\rm Re}\left(T_f(z)\right)>0$ for $|z|<\mathrm{R_{\alpha,\beta,{Co(A)}}}$, where $\mathrm{R_{\alpha,\beta,{Co(A)}}}$ is the least value of $r\in(0,1)$ satisfying $\Phi_2(r)=0$ with
	\begin{align*}
	\Phi_2(r)=(A+3+2\alpha-2\beta)r^2-2(A+1+\alpha+\beta)r+A-1.
	\end{align*}
	The radius $\mathrm{R_{\alpha,\beta,{Co(A)}}}$ is sharp.
\end{theorem}
	\begin{proof}
		In view of \cite[Proof of Theorem 3.2, p. 276]{Ravi-Ron-Shan-CVEE-2007}, we have
		\begin{align*}
			\bigg|1+\frac{z f^{\prime\prime}(z)}{f^{\prime}(z)}-\frac{1+(\beta-\alpha-1)r^2}{1-r^2}\bigg|\leq\frac{(\alpha+\beta)r}{1-r^2},
		\end{align*}
		\vspace{1.2mm}

		which yields that
		\begin{align}\label{Eq-3.1}
			{\rm Re}\left(1+\frac{zf^{\prime\prime}(z)} {f^{\prime}(z)}\right)\leq\frac{1+(\alpha+\beta)r+(\beta-\alpha-1)r^2}{1-r^2}.
		\end{align}
			Thus, by the inequality \eqref{Eq-3.1}, for $0\leq\alpha\leq\beta$, we get
		\begin{align*}
			{\rm Re}\;\left(T_{f(z)}\right)&\geq \frac{2}{A-1}\left[\frac{A+1}{2}\frac{1-r}{1+r}-\frac{1+(\alpha+\beta)r+(\beta-\alpha-1)r^2}{1-r^2}.\right]\\&=\frac{1}{(A-1)(1-r^2)}\left[(A+3+2\alpha-2\beta)r^2-2(A+1+\alpha+\beta)r+A-1\right]\\&=\frac{1}{(A-1)(1-r^2)}\Phi_2(r),
		\end{align*}
		where $\Phi_2$ is given in the statement of theorem.\vspace{2mm}
			
		The right hand side of the above inequality is strictly positive if $|z|<\mathrm{R_{\alpha,\beta,{Co(A)}}}$, where  $\mathrm{R_{\alpha,\beta,{Co(A)}}}$ is given in the statement of the theorem. We now investigate the existence of the root $\mathrm{R_{\alpha,\beta,{Co(A)}}}\in (0, 1)$ for each $A\in(1,2]$.\vspace{1.2mm} 
		
		We see that the function $\Phi_2(r)$ which is defined in the statement of the theorem is continuous on $[0,1]$ with
		\begin{align*}
			\Phi_2(r)(0)=A-1>0\;\;\;\mbox{and}\;\;\;\Phi_2(r)(1)=-4\beta<0;\;\;\mbox{for}\;\beta>0.
		\end{align*}

		By the IVT, $\Phi_2(r)$ has at least one root in $(0,1)$. \\Hence, 	${\rm Re}\left(T_{f(z)}\right)>0$ if $|z|=r<\mathrm{R_{\alpha,\beta,{Co(A)}}}$ exists for every $A\in(1,2]$.Also, if we consider $f_0(z)=\frac{z}{(1-z)^{\beta-\alpha}}$, $z\in\mathbb{D}$, then for this function we compute
		\begin{align*}
			T_{f_0}(z)=\frac{2}{A-1}\left[\frac{A+1}{2}\frac{1-z}{1+z}-1+\frac{2(\alpha-\beta)}{(1-z)^{1-\alpha+\beta}}-\frac{(-1+\alpha-\beta)(\alpha-\beta)z}{(1-z)^{2-\alpha+\beta}}\right].
		\end{align*} 
		We observe that, if $z=-r$ and $\mathrm{R_{\alpha,\beta,{Co(A)}}}<|z|<1$, then ${\rm Re}T_{f_0}(z)<0$ (see figure). This proves the sharpness of the radius of concavity. This completes the proof.
	\end{proof}
\section{\bf Radius of concavity for functions in the class $\mathcal{\tilde{S^*}(\beta)}$}
A function $f\in\mathcal{A}$ is said to be strongly starlike of order $\beta$ in $\mathbb{D}$ if it satisfies the relation
\begin{align*}
	\frac{ zf^{\prime}(z)}{f(z)}\prec \left(\frac{1+z}{1-z}\right)^{\beta}
\end{align*}
for some real $\beta$ $(0\leq \beta\leq 1)$. We denote this class by $ \mathcal{\tilde{S^*}(\beta)} $. Note that $\mathcal{\tilde{S^*}}(1)=\mathcal{S}^{*}$. \vspace{1.2mm}

For the class $\mathcal{\tilde{S^*}(\beta)}$, we determine the radius of concavity and prove the following result.
\begin{theorem}
	If $f \in\mathcal{\tilde{S^*}(\beta)}$, then ${\rm Re}\left(T_f(z)\right)>0$ for $|z|<\mathrm{R_{\beta,{Co(A)}}}$, where $\mathrm{R_{\beta,{Co(A)}}}$ is the least value of $r\in(0,1)$ satisfying $\Phi_3(r)=0$ with 
	\begin{align*}
		\Phi_3(r)=(A+1)\frac{1-r}{1+r}-\frac{4r\beta}{1-r^2}-2\left(\frac{1+r}{1-r}\right)^{\beta}.
	\end{align*}
	The radius $\mathrm{R_{\beta,{Co(A)}}}$ is sharp.
\end{theorem}
	\begin{proof}
	    	We have
		\begin{align}\label{Eq-4.1}
			\frac{zf^{\prime}(z)}{f(z)}= \left(\frac{1+g(z)}{1-g(z)}\right)^{\beta}.
		\end{align}
	An easy computation from \eqref{Eq-4.1} shows that
		\begin{align*}
			1+\frac{zf^{\prime\prime}(z)}{f(z)}=\left(\frac{1+g(z)}{1-g(z)}\right)^{\beta}+2\beta\frac{zg^{\prime}(z)}{1-\left(g(z)\right)^2},
		\end{align*}
		where $g$ is analytic in $\mathbb{D}$ satisfying $|\phi(z)|\leq1$ for $z\in\mathbb{D}$.\vspace{1.2mm}
		
	\noindent From the Schwarz-pick lemma, we get
		\begin{align*}
	|g^{\prime}(z)|\leq\frac{1-|g(z)|^2}{1-|z|^2}
			\end{align*}
			which implies that
			\begin{align*}
				\frac{|g^{\prime}(z)}{1-|g(z)|^2}\leq\frac{1}{1-|z|^2}\;\;\mbox{for}\;\;z\in\mathbb{D}.	
			\end{align*}
			This shows that
			\begin{align*}
				{\rm Re}\left(1+\frac{zf^{\prime\prime}(z)} {f^{\prime}(z)}\right)\leq\left(\frac{1+r}{1-r}\right)^{\beta}+\frac{2r\beta}{1-r^2}.	
			\end{align*}
		In view of the above inequality, we get
			\begin{align*}
				{\rm Re}\;\left(T_{f}(z)\right)&\geq\frac{2}{A-1}\left[\frac{A+1}{2}\frac{1-r}{1+r}-1-\frac{zf^{\prime\prime}(z)} {f^{\prime}(z)}\right]\\&=\frac{1}{A-1}\left[(A+1)\frac{1-r}{1+r}-\frac{4r\beta}{1-r^2}-2\left(\frac{1+r}{1-r}\right)^{\beta}\right].
			\end{align*}
			Let us define a function $\Phi_3(r)$, where 
			\begin{align*}
			\Phi_3(r)=(A+1)\frac{1-r}{1+r}-\frac{4r\beta}{1-r^2}-2\left(\frac{1+r}{1-r}\right)^{\beta}.
			\end{align*}
			Thus, $	{\rm Re}\;\left(T_{f(z)}\right)>0$ if $|z|<\mathrm{R_{\beta,{Co(A)}}}$. We now investigate the existence of $r$ for each $\beta \in(0,1]$ and $A\in(1,2]$. We see that the function $\Phi_3(r)$ is continuous on $[0,1]$ with
			\begin{align*}
				\Phi_3(0)=A-1>0\;\;\mbox{and}\;\;\lim_{r\rightarrow {1}^{-}}\Phi_3(r)(r)=-\infty.
			\end{align*}
			Thus by the IVT, we see that $\Phi_3(r)$ has at least one root in $(0,1)$. Hence, 	${\rm Re}\left(T_{f(z)}\right)>0$ if $|z|=r<\mathrm{R_{\beta,{Co(A)}}}$ exists for every $A\in(1,2]$. \vspace{1.2mm}
			
			Next part of the proof is to show that the radius $\mathrm{R_{\beta,{Co(A)}}}$ is sharp. Henceforth, for $0<\lambda<n(n+1)$, $n\in\mathbb{N}$, we consider (see \cite[p. ]{Xu-Yang-JIA-2015}) the function
			\begin{align*}
				f_{\lambda}(z)=\frac{\lambda}{n(n+1)}z^{n+1},\; z\in\mathbb{D}.
			\end{align*} 
			Then for this function, a simple computation yields that
			\begin{align*}
				T_{f_{\lambda}}(z)=\frac{2}{A-1}\left[\frac{A+1}{2}\left(\frac{1-z}{1+z}\right)-\frac{1+(n+1)\lambda z^n}{(1+\lambda z^n)}\right].
			\end{align*} 
				\vspace{1.2mm}

			We see that, if $z=-r$ and $\mathrm{R_{\beta,{Co(A)}}}<|z|<1$, then it can be shown that ${\rm Re}\;T_{f_{\lambda}}(z)<0$. This proves the sharpness of the radius $\mathrm{R_{\beta,{Co(A)}}}$. 
	\end{proof}
\section{\bf Radius of concavity for functions in the class $\mathcal{S}^*(\alpha)$}
A function $f(z)=z+\sum_{n=2}^{\infty}a_nz^n$ analytic in $\mathbb{D}$ is said to be starlike of order $\alpha$, where $\alpha$ is fixed and $0\leq\alpha<1$ if 
\begin{align}
	{\rm Re}\left(\frac{zf^{\prime}(z)}{f(z)}\right)>\alpha
\end{align}
for all $z$ in $\mathbb{D}$. We denote, for fixed $\alpha$ the class of all starlike functions of order $\alpha$ by $\mathcal{S^*(\alpha)}$. Note that for $\alpha=0$, the class is called starlike function and denoted by $\mathcal{S^*}$ or frequently say class of starlike functions of order zero. \vspace{1.2mm}

We obtained the following result by finding the radius of concavity for the class $\mathcal{S}^*(\alpha)$. 
\begin{theorem}
If $f\in\mathcal{S}^*(\alpha)$, then ${\rm Re}\;\left(T_{f(z)}\right)>0$ for $|z|<\mathrm{R_{\alpha,{Co(A)}}}$, where $\mathrm{R_{\alpha,{Co(A)}}}$ is the least value of $r\in(0,1)$ satisfying $\Phi_4(r)=0$ with
\begin{align*}
	\Phi_4(r)=(-8\alpha^2+6\alpha-2A\alpha+A-1)r^2+(2A\alpha-2A-10\alpha+6)r+A-1.
\end{align*}
\end{theorem}
The radius $\mathrm{R_{\alpha,{Co(A)}}}$ is sharp.
\begin{proof}
		In view of \cite[p-67]{Schild-AJM-1965}, we see that if $-1<b<1$, then 
		\begin{align*}
				&1+\frac{zf^{\prime\prime}(z)}{f(z)}\\&=\frac{(2\alpha-1)^2z^4+2b(1+\alpha-4\alpha^2)z^3+2(-3+2\alpha+2\alpha^2b^2)z^2+2b)(1-3\alpha)z+1}{(1-2bz+z^2)[1-2\alpha bz+(2\alpha-1)z^2]}
		\end{align*}
	 and if $b=-1$, $\phi(z)=1$
	 \begin{align*}
	 		1+\frac{zf^{\prime\prime}(z)}{f(z)}=\frac{(2\alpha-1)^2z^2+(6\alpha-4)z+1}{(1+z)[1+(2\alpha-1)z]}.
	 \end{align*}
	 By the above inequality, we get
	 \begin{align*}
	 	{\rm Re}\;T_{f}(z)&\geq\frac{2}{A-1}\left[\frac{A+1}{2}\frac{1-r}{1+r}-\frac{(2\alpha-1)^2r^2+(6\alpha-4)r+1}{(1+r)[1+(2\alpha-1)r]}\right]\\&=\frac{\Phi_4(r)}{(A-1)(1+(2\alpha-1)r)},
	 \end{align*}
	  where 
	 \begin{align*}
	 	 \Phi_4(r)=(-8\alpha^2+6\alpha-2A\alpha+A-1)r^2+(2A\alpha-2A-10\alpha+6)r+A-1.
	 \end{align*}
	 	Thus, ${\rm Re}\;\left(T_{f}(z)\right)>0$ if $|z|<\mathrm{R_{\alpha,{Co(A)}}}$. We now investigate the existence of r for each $\alpha\in[0,1]$ and $A\in(1,2]$. We see that the function $\Phi_4(r)$ is continuous on $[0,1]$ with
	 	\begin{align*}
	 		 \Phi_4(0)=A-1>0\;\;\mbox{and}\;\; 	 \Phi_4(1)=-8\alpha^2-4\alpha+4<0 \;\;\mbox{if}\;\;\alpha>1/2.
	 	\end{align*}
	 By the IVT, $\Phi_4(r)$ has at least one root in $(0,1)$.\\ Hence, 	${\rm Re}\left(T_{f(z)}\right)>0$ if $|z|=r<\mathrm{R_{\alpha,{Co(A)}}}$ exists for every $A\in(1,2]$. \vspace{1.2mm}
	 
	 In order to show that the radius $\mathrm{R_{\alpha,{Co(A)}}}$ is sharp, we consider the function  $f_{\alpha}$ (see \cite[p. 70]{Schild-AJM-1965})
	 \begin{align*}
	 	f_{\alpha}(z)=\frac{z}{(1-2bz+z^2)^{1-\alpha}},\; z\in\mathbb{D}.
	 \end{align*}
	 We observe that, if $z=-r$ and $\mathrm{R_{\alpha,{Co(A)}}}<|z|<1$, then it can be shown that ${\rm Re}\;T_{f_{\alpha}}(z)<0$. This proves the sharpness of the radius of concavity.
\end{proof}
\section{\bf Radius of concavity for a certain class of analytic functions}
 Let $f(z)=z+a_2z^2+....$ be analytic in the unit disc $\mathbb{D}:|z|<1$. If there exists a function $s(z)$, univalent and starlike with respect to the origin in $\mathbb{D}$, such that 
 \begin{align}\label{Eq-6.1}
 {\rm Re}\left(\frac{f(z)}{s(z)} \right)>0
 \end{align}
 holds in $\mathbb{D}$, then $f(z)$ is said to be close-to-star due to Reade (see \cite{Reade-MMJ-1955-56}).
 The radius of starlikeness (and univalence) of the class of close-to-star functions is $(2-\sqrt{3})$; this result is due to MacGregor (see \cite[p. 517]{Macgregor-Proc.AMS-1963}) and Sakaguchi (see \cite[p. 208]{Sakaguchi-JMSJ-1963}). The radius of convexity for that class is $(5-2\sqrt{6})$; this result is due to Sakaguchi (see \cite[p.6]{Sakaguchi-JNGU-1964}).\vspace{1.2mm}
 
 In this note we consider a subclass of the class of functions close-to-star in $\mathbb{D}$; the members of the subclass are those analytic functions $f(z)=z+....$ which satisfy \eqref{Eq-6.1} with $s(z)\equiv z$, \textit{i.e.,}
\begin{align}\label{Eq-6.2}
	{\rm Re}\left(\frac{f(z)}{z}\right)>0
\end{align}
holds in $\mathbb{D}$. MacGregor (see \cite{Macgregor-AMS-1962}) has proved that the radius of starlikeness (and univalence) of this subclass is ($\sqrt{2}-1$) . We shall determine the radius of concavity of this particular class of close-to-star functions.\vspace{1.2mm}

\begin{thmA}\emph{(see\cite{Reade-Ohawa-Sakaguchi-JNGU-1965})}
	If $f(z)=z+....$is analytic in $\mathbb{D}$ and satisfies \eqref{Eq-6.2}, then $f(z)$ is (univalent and) convex for
	\begin{align}\label{Eq-6.4}
		|z|<r_0 = 0.179...,
	\end{align}
	where $r_0$ is the smallest positive root of the equation 
	\begin{align*}
		1-5r-3r^2-r^3=0.
	\end{align*}
	The radius $r_0$ is sharp.
\end{thmA}
We obtained the following result by finding the radius of concavity for the class of analytic functions. 

\begin{theorem}\label{Th-6.1}
	If $f(z)=z+a_2z^2+....$is analytic in $\mathbb{D}$ and satisfies \eqref{Eq-6.2} then ${\rm Re}\;\left(T_{f(z)}\right)>0$ for $|z|<\mathrm{R_{Co(A)}}<1$, where $\mathrm{R_{Co(A)}}$ is the least value of $r\in(0, \sqrt{2}-1)$ satisfying $\Phi_6(r)=0$, where
	\begin{align*}
		\Phi_6(r)=-1+A-(4A+8)r+(A+21)r^2-20r^3-(A+11)r^4.
	\end{align*}
	The radius $\mathrm{R_{Co(A)}}$ is sharp.
\end{theorem}
The following the lemma will be helpful to prove the result.
\begin{lemA}\emph{(see\cite{Reade-Ohawa-Sakaguchi-JNGU-1965})}
	If $f(z)=z+....$is analytic in $\mathbb{D}$ and if $f(z)$ satisfies \eqref{Eq-6.2} there, then 
	\begin{align}\label{Eq-6.3}
		\bigg|\frac{zf^{\prime}(z)}{f(z)}\bigg|\geq\frac{1-2r-r^2}{1-r^2};\;\;\;\;|z|=r,
	\end{align}
	holds in $\mathbb{D}$.
\end{lemA}
\begin{proof}[\bf Proof of Theorem \ref{Th-6.1}]
	Without any loss of generality, we may assume that $f$ is analytic and satisfies ${\rm Re}\; \left(f(z)/z\right)>0$ in $\overline{\mathbb{D}}$, the closure of $\mathbb{D}$. Considering the function (see \cite[p.2]{Reade-Ohawa-Sakaguchi-JNGU-1965})
	\begin{align*}
		g(z)&=\left(f\left(\frac{z+\alpha}{1+\bar{\alpha}z}\right)\bigg/\frac{z+\alpha}{1+\bar{\alpha}z}\right)\frac{(a+\alpha)(1+\bar{\alpha}z)}{z}\\&=\frac{f(\alpha)}{z}+\left(f^{\prime}(\alpha)(1-|\alpha|^2)+f(\alpha)2\bar{\alpha}\right)\\&\quad+\left(f^{\prime\prime}(\alpha)\frac{(1-|\alpha|^2)^2}{2}+f^{\prime}(\alpha)\bar{\alpha}\left(1-|\alpha|^2\right)+f(\alpha)\bar{\alpha}^2\right)z+\cdots
	\end{align*}
	for an arbitrary $\alpha$ with $|\alpha|<1$. Since $(z+\alpha)(1+\bar{\alpha}z)/z$ is real and positive for $|z|=1$, it follows that ${\rm Re}\; g(z)>0$ for $|z|=1$. hence, by a lemma (see \cite[p. 514]{Robrtson-DMJ-1939}), replacing $\alpha$ by $z$ and \eqref{Eq-6.3}, it can be obtained that
	\begin{align*}
		\bigg|\frac{zf^{\prime\prime}(z)} {f^{\prime}(z)}+\frac{2r^2}{1-r^2}\bigg|\leq\frac{4r-6r^2+4r^3+2r^4}{(1-r^2)(1-2r-r^2)}
	\end{align*}
	where $z$ satisfies $|z|=r< (\sqrt{2}-1)$ in order that $1-2r-r^2>0$ be valid. \vspace{1.2mm}
	
	Thus it follows that
		\begin{align*}
			{\rm Re}\left(1+\frac{zf^{\prime\prime}(z)} {f^{\prime}(z)}\right)&\leq 1-\frac{2r^2}{1-r^2}+\frac{4r-6r^2+4r^3+2r^4}{(1-r^2)(1-2r-r^2)}\\&=\frac{5r^4+10r^3-10r^2+2r+1}{(1-r^2)(1-2r-r^2)}\;\mbox{for}\; |z|=r<\sqrt{2}-1.		
    	\end{align*}
    	Consequently, the above inequality gives that
    	\begin{align*}
    		{\rm Re}\left(T_{f}(z)\right)&\geq \frac{2}{A-1}\left[\frac{A+1}{2}\left(\frac{1-r}{1+r}\right)-\frac{5r^4+10r^3-10r^2+2r+1}{(1-r^2)(1-2r-r^2)}\right]\\&=\frac{-(A+11)r^4-20r^3+(A+21)r^2-(4A+8)r+A-1}{(A-1)(1-r^2)(1-2r-r^2)}\\&=\frac{\Phi_6(r)}{(A-1)(1-r^2)(1-2r-r^2)},
    	\end{align*}
    	where $\Phi_6$ is given in the statement of theorem.\vspace{2mm}
    	
    	The right hand side of the above inequality is strictly positive if $|z|<\mathrm{R_{Co(A)}}$, where  $\mathrm{R_{Co(A)}}$ is given in the statement of the theorem. We now investigate the existence of the root $\mathrm{R_{Co(A)}}\in (0, (\sqrt{2}-1))$ for each $A\in(1,2]$.\vspace{1.2mm} 
    	
    	We see that the function $\Phi_6(r)$ which is defined in the statement of the theorem is continuous on $[0,(\sqrt{2}-1)]$ with
    	\begin{align*}
    		\Phi_6(r)(0)=A-1>0\;\;\;\mbox{and}\;\;\;\Phi_6(r)(\sqrt{2}-1)<0.
    	\end{align*}
    	By the IVT, $\Phi_6(r)$ has at least one root in $(0,(\sqrt{2}-1))$. Hence, 	${\rm Re}\left(T_{f(z)}\right)>0$ if $|z|=r<\mathrm{R_{Co(A)}}$ exists for every $A\in(1,2]$.\vspace{1.2mm} 
    	
    	To show that the radius is best possible, we consider the function (see \cite[p. 2]{Reade-Ohawa-Sakaguchi-JNGU-1965}) 
    	\begin{align*}
    		f(z)=\frac{z(z+1)}{1-z},\; z\in\mathbb{D}.
    	\end{align*}
    	Then for this function, we have
    	\begin{align*}
    		T_{f}(z)=\frac{2}{A-1}\left[\frac{A+1}{2}\left(\frac{1-z}{1+z}\right)-1-\frac{7z-z^2}{(1-z)^3}\right].
    	\end{align*} 
    If $z=-r$ and $\mathrm{R_{Co(A)}}<|z|<1$, then a simple computation shows that ${\rm Re}\;T_{f}(z)<0$. This proves the sharpness of the radius $\mathrm{R_{Co(A)}}$.
    \vspace{1.2mm}

\end{proof}
\noindent\textbf{Compliance of Ethical Standards:}\\

\noindent\textbf{Conflict of interest.} The authors declare that there is no conflict  of interest regarding the publication of this paper.\vspace{1.5mm}

\noindent\textbf{Data availability statement.}  Data sharing is not applicable to this article as no datasets were generated or analyzed during the current study.

    \end{document}